\documentclass[a4paper,12pt]{article}

\usepackage[mathscr]{eucal}

\usepackage{amscd}
\usepackage{amssymb,amsfonts,amsmath,amsthm}
\usepackage{verbatim}

\usepackage{amsmath}
\usepackage{amsthm}
\usepackage{amssymb}

\usepackage{latexsym}
\usepackage{array}
\usepackage{enumerate}

\numberwithin{equation}{section}

\newcommand{\CC}{\mathbb{C}}

\newcommand{\QQ}{\mathbb{Q}}
\newcommand{\ZZ}{\mathbb{Z}}
\newcommand{\D}{\mathcal{D}}
\newcommand{\E}{\mathcal{E}}

\newcommand{\G}{\mathcal{G}}
\newcommand{\HH}{\mathcal{H}}
\newcommand{\II}{\mathcal{J}}
\newcommand{\LL}{\mathcal{L}}
\newcommand{\M}{\mathcal{M}}
\newcommand{\N}{\mathcal{N}}
\newcommand{\K}{\mathcal{K}}
\newcommand{\sho}{\mathcal{O}}
\newcommand{\R}{\mathcal{R}}
\newcommand{\SH}{\mathcal{S}}

\newcommand{\bfD}{\mathbf{D}}

\renewcommand{\dim}{{\rm dim}}

\newcommand{\e}{\varepsilon}

\newcommand{\simto}{\overset{\sim}{\longrightarrow}}

\newcommand{\fin}{\hspace*{\fill}$\Box$\vspace*{2mm}}

\newtheorem{theorem}{Theorem}[section]

\newtheorem{corollary}[theorem]{Corollary}
\newtheorem{lemma}[theorem]{Lemma}
\newtheorem{proposition}[theorem]{Proposition}

\theoremstyle{definition}
\newtheorem{definition}[theorem]{Definition}
\theoremstyle{remark}
\newtheorem{remark}[theorem]{\sc Remark}
\title{On a Bernstein-Sato polynomial 
of a meromorphic function 
\footnote{
Bernstein-Sato polynomials, meromorphic 
functions, monodromy, multiplier ideal sheaves}
\\ \small To the memory of Professor Hikosaburo Komatsu}

\author{Kiyoshi TAKEUCHI 
\footnote{Mathematical Institute, Tohoku University, 
Aramaki Aza-Aoba 6-3, Aobaku, Sendai, 980-8578, Japan. 
E-mail: takemicro@nifty.com} 
}

\date{}

\begin{document}

\maketitle

\begin{abstract}
We define Bernstein-Sato 
polynomials for meromorphic functions 
and study their basic properties. In particular, 
we prove a Kashiwara-Malgrange type 
theorem on their geometric monodromies, which would 
be useful also in relation with the monodromy conjecture. 
A new feature in the meromorphic setting is that we have 
several b-functions whose roots yield the same 
set of the eigenvalues of the Milnor monodromies. We introduce 
also multiplier ideal sheaves for meromorphic functions 
and show that their jumping numbers are 
related to our b-functions. 
\end{abstract}

\maketitle

\section{Introduction}\label{sec:s1}

The theory of b-functions initiated by 
Bernstein and Sato independently is 
certainly on a crossroad of various branches  
of mathematics, such as generalized functions, 
singularity theory, prehomogenous vector 
spaces, D-modules, number theory, 
algebraic geometry, computer algebra and so on. 
We often call them 
Bernstein-Sato polynomials. 
To see the breadth of their influence to mathematics, 
we can now consult for example the excellent 
survey articles by \cite{A-J-N} and \cite{Budur}. 

\par
 Let us briefly recall the definitions of 
classical Bernstein-Sato polynomials 
and some related results. For this purpose, 
let $X$ be a complex manifold and 
$\sho_X$ the sheaf of holomorphic functions 
on it. Denote by $\D_X$ the sheaf of 
differential operators with holomorphic 
coefficients on $X$. Let $f \in \sho_X$ be a holomorphic 
function defined on a neighborhood of 
a point $x_0 \in X$ such that $f(x_0)=0$. Then the (local) 
Bernstein-Sato polynomial $b_f(s) \in \CC [s]$ 
of $f$ (at $x_0 \in X$) is the non-zero 
polynomial $b(s) \not= 0$ of the lowest 
degree satisfying the equation 
\begin{equation}
b(s) f^{s} = P(s) f^{s+1} 
\end{equation}
for some $P(s) \in \D_X [s]$. 
In the algebraic and analytic cases, 
the existence of such $b(s) \not= 0$ 
was proved by Bernstein and Bj\"ork 
respectively. Then  
Kashiwara \cite{K-1} proved that the roots of 
the Bernstein-Sato polynomial 
$b_f(s)$ are negative rational numbers. 
One of the most strinking results on $b_f(x)$ 
is the Kashiwara-Malgrange theorem 
in \cite{K-2} and \cite{M}, which 
asserts that the set of the eigenvalues 
of the local (Milnor) monodromies of 
$f$ at various points $x \in f^{-1}(0)$ 
close to $x_0 \in f^{-1}(0)$ is equal to 
the one $\{ \exp (2 \pi i \alpha ) \ | \ 
\alpha \in (b_f)^{-1}(0) \}$. Motivated by it, 
Denef and Loeser fomulated their celebrated monodromy 
conjecture in \cite{D-L}. 
Later in \cite{Sabbah-1} and \cite{Sabbah-2}, Sabbah 
developed a theory of b-functions of 
several variables. More precisely, he considered 
several holomorphic functions 
$f_1, f_2, \ldots, f_k \in \sho_X$ 
($k \geq 1$) and proved the existence 
of a non-zero polynomial 
$b(s) \in \CC [s]= \CC [s_1,s_2, \ldots, s_k]$ of 
$k$ variables $s=(s_1, s_2, \ldots, s_k)$ 
satisfying the equation 
\begin{equation}
b(s) \Bigl( \prod_{i=1}^k f_i^{s_i} \Bigr) 
= P(s) \Bigl( \prod_{i=1}^k f_i^{s_i+1} \Bigr) 
\end{equation}
for some $P(s) \in \D_X [s] = 
\D_X [s_1,s_2, \ldots, s_k]$. See also 
Gyoja \cite{G} for a different proof 
and some additional results. 
The non-zero ideal $I \subset \CC [s]$ thus 
obtained is now called the Bernstein-Sato 
ideal of $f=(f_1, f_2, \ldots, f_k)$. The 
geometric meaning of this $I$ was clarified 
only recently in 
Budur-van der Veer-Wu-Zhou \cite{B-V-W-Z}. 
Moreover by Budur-Mustata-Saito \cite{B-M-S}, 
the theory of b-functios has been also generalized 
to higher-codimensional subvarieties, i.e to 
arbitrary ideals $\II \subset \sho_X$ of $\sho_X$. 
Their b-functions are related to the monodromies of 
the Verdier specializations along $\II$. 
See \cite{Budur} and \cite{B-M-S} for the details. 

\par 
The aim of this short note is to define Bernstein-Sato 
polynomials for meromorphic functions 
and study their basic properties. 
For two holomorphic functions 
$F, G \in \sho_X$ such that $F \not= 0$, 
$G \not= 0$ defined on a neighborhood of 
a point $x_0 \in X$ and coprime to each other 
such that $F(x_0)=0$, 
let us consider the meromorphic function 
\begin{equation}
f(x)= \frac{F(x)}{G(x)}
\end{equation}
associated to them. Let 
$D=F^{-1}(0) \cup G^{-1}(0) \subset X$ 
be the divisor defined by $F \cdot G \in \sho_X$ and 
\begin{equation}
\sho_X \Bigl[ \frac{1}{FG} \Bigr] 
= \Bigl\{ \frac{h}{(FG)^l} \ | \ 
h \in \sho_X, l \geq 0 \Bigr\}
\end{equation} 
the localization of $\sho_X$ along $D \subset X$. 
Recall that this sheaf is endowed with the 
structure of a left $\D_X$-module. Then the 
polynomial ring $( \sho_X [ \frac{1}{FG} ] )[s]$ 
over it is naturally a left $\D_X [s]$-module. 
As in the classical case where $G=1$ and $f$ is 
holomorphic, on the rank-one free module 
\begin{equation}
\LL := \Bigl( \sho_X \Bigl[ \frac{1}{FG} \Bigr] 
\Bigr) [s] f^s 
 \simeq  \Bigl( \sho_X \Bigl[ \frac{1}{FG} \Bigr] 
\Bigr) [s]
\end{equation}
over it, we define naturally a structure of 
a left $\D_X [s]$-module and can consider 
its $\D_X [s]$-submodule $\D_X [s] f^s 
\subset \LL$ generated by $f^s \in \LL$. 
However, in order to prove a Kashiwara-Malgrange type 
theorem (see Theorem \ref{Main-2} below) 
for b-functions on the geometric 
monodromies of $f$ in our meromorphic setting, 
we have to consider also other types of 
$\D_X [s]$-submodules of $\LL$. Considering 
\begin{equation}
\frac{1}{G^m} f^{s+k} = 
\frac{f^k}{G^m} \cdot f^{s} \in \LL 
\end{equation}
for various integers $m \geq 0$ and $k \geq 0$, 
we obtain the following result. 

\begin{theorem}\label{Main-1} 
Let $m \geq 0$ be a non-negative integer. Then 
there exists a non-zero polynomial $b(s) \in \CC [s]$ 
such that 
\begin{equation}\label{fe-1}
b(s) \Bigl( \frac{1}{G^m} f^{s} \Bigr) \in 
\sum_{k=1}^{+ \infty} \D_X [s]
\Bigl( \frac{1}{G^m} f^{s+k} \Bigr) 
\end{equation} 
i.e. there exist 
$P_1(s), P_2(s), \ldots, P_N(s) \in \D_X [s]$ 
for which we have 
\begin{equation}
b(s) \Bigl( \frac{1}{G^m} f^{s} \Bigr) = 
\sum_{k=1}^{N} P_k(s) 
\Bigl( \frac{1}{G^m} f^{s+k} \Bigr). 
\end{equation}
\end{theorem}

Although the proof of Theorem \ref{Main-1} 
relies on the classical theory of Kashiwara 
and Malgrange, we need some new ideas to formulate 
and prove it. See Section \ref{sec:s2} 
for the details.  This could be the reason 
why Bernstein-Sato polynomials for meromorphic functions 
were not defined nor studied before. 

\begin{definition}\label{BS-fun}
For $m \geq 0$ we denote by $b_{f,m}^{{\rm mero}}(s) 
\in \CC [s]$ the minimal polynomial (i.e. the 
non-zero polynomial of the lowest degree) 
satisfying the equation in Theorem \ref{Main-1} 
and call it the Bernstein-Sato polynomial 
or the b-function of $f$ of order $m$. 
\end{definition}

By a theorem of Sabbah \cite{Sabbah-1} there 
exists a non-zero polynomial $b(s_1,s_2) \not= 0$ of 
two variables $s_1, s_2$ such that 
\begin{equation}
b(s_1,s_2) F^{s_1}G^{s_2} = 
P(s_1,s_2) F^{s_1+1}G^{s_2+1}
\end{equation}
for some $P(s_1,s_2) \in \D_X [s_1,s_2]$. 
Then by setting $s_1=s$ and  $s_2=-s-m-2$ we obtain 
the desired condition 
\begin{equation}
b(s, -s-m-2) \Bigl( \frac{1}{G^m} f^{s} \Bigr) = 
G^2 P(s, -s-m-2)
\Bigl( \frac{1}{G^m} f^{s+1} \Bigr). 
\end{equation}
This important remark is due to Oaku. 
However, for the given $F(x), G(x) \in \sho_X$ 
it would not be so easy to verify that 
the polynomial $b(s, -s-m-2) \in \CC [s]$ 
of $s$ thus obtained is non-zero. Recall 
that by Bahloul \cite{B-1}, \cite{B}, Bahloul-Oaku 
\cite{B-O}, Oaku-Takayama \cite{O-T} and 
Ucha-Castro \cite{U-C} 
we have algorithms to compute 
the Bernstein-Sato ideal $I \subset \CC [s_1,s_2]$ 
at least when $F$ and $G$ are polynomials. Motivated 
by this observation, instead of 
the equation \eqref{fe-1} 
one may also consider the simpler one 
\begin{equation}\label{fe-2}
b(s) \Bigl( \frac{1}{G^m} f^{s} \Bigr) \in 
\D_X [s] \Bigl( \frac{1}{G^m} f^{s+1} \Bigr). 
\end{equation}
Then of course, the minimal polynomial 
$b(s) \not= 0$ satisfying it is divided by 
our b-function $b_{f,m}^{{\rm mero}}(s)$, 
but from the proof of Theorem \ref{Main-2} below 
it looks that we do not have a 
Kashiwara-Malgrange type result as in it  
by this simpler 
definition of b-functions. This explains 
the reason why the right hand side of 
the equation \eqref{fe-1} is not so simple. 
Note also that if $G=1$ and $f= \frac{F}{G}=F$ is holomorphic 
we have $f \in \sho_X \subset \D_X$ and for any $m \geq 0$ 
our b-function $b_{f,m}^{{\rm mero}}(s)$ coincides 
with the classical one $b_f(s) \in \CC [s]$ 
introduced by Bernstein and Sato. 
But in the meromorphic case, the relation among 
$b_{f,m}^{{\rm mero}}(s)$ for various 
$m \geq 0$ is not very clear so far. 
See Lemma \ref{Lemma-4} for a weak relation 
among their roots.  Nevertheless, we can prove a 
Kashiwara-Malgrange type result as follows. 
First, recall the following theorem due to \cite{G-L-M}. 

\begin{theorem}\label{the-fib} 
(Gusein-Zade, Luengo and Melle-Hern\'andez 
\cite{G-L-M}) 
For any point $x \in F^{-1}(0)$ 
close to the point $x_0$ there exists 
$\e_0> 0$ such that for any $0< \e < \e_0$ 
and the open ball $B(x; \e ) \subset X$ of 
radius $\e >0$ with center at $x$ 
(in a local chart of $X$) the restriction 
\begin{equation}
B(x; \e ) \setminus G^{-1}(0) 
\longrightarrow \CC
\end{equation}
of $f: X \setminus G^{-1}(0) 
\longrightarrow \CC$ is a locally trivial 
fibration over a sufficiently small 
punctured disk in $\CC$ with center at 
the origin $0 \in \CC$ 
\end{theorem}

We call the fiber in this theorem the Milnor fiber of 
the meromorphic function $f(x)= \frac{F(x)}{G(x)}$ 
at $x \in F^{-1}(0)$ and denote it by $M_x$. 
As in the holomorphic case (see Milnor \cite{Milnor}), 
we obtain also its Milnor monodromy operators 
\begin{equation}
\Phi_{j,x}: H^j(M_x; \CC ) \simto H^j(M_x; \CC ) \qquad 
(j \geq 0 ). 
\end{equation}
Then we have the following result. 
Let ${\rm E}_{f,x_0} \subset \CC^*$ be the set of 
the eigenvalues 
of the monodromies $\Phi_{j,x}$ of $f$ at the points 
$x \in F^{-1}(0)$ close to $x_0$ and $j \geq 0$. 

\begin{theorem}\label{Main-2} 
Let $m \geq 0$ be a non-negative integer. Then 
we have 
\begin{equation}
\{ \exp (2 \pi i \alpha ) \ | \ 
\alpha \in (b_{f,m}^{{\rm mero}})^{-1}(0) \} 
\subset {\rm E}_{f,x_0}. 
\end{equation}
If we assume moreover that $m \geq 2 \dim X$, 
then we have an equality 
\begin{equation}
\{ \exp (2 \pi i \alpha ) \ | \ 
\alpha \in (b_{f,m}^{{\rm mero}})^{-1}(0) \} 
= {\rm E}_{f,x_0}. 
\end{equation}
\end{theorem}

Combining Theorem \ref{Main-2} with the results in 
\cite{N-T} and \cite{Raibaut}, one may 
formulate a monodromy conjecture for 
rational functions, like the original 
one in \cite{D-L}. For previous works in 
this direction, see for example \cite{G-L} and \cite{V-Z}. 
Note that if in a coordinate system 
$F$ and $G$ depend on 
separated variables we can easily see that 
our $b_{f,m}^{{\rm mero}}(s)$ coincides with 
the b-function $b_F(s)$ of the holomorphic 
function $F$. At the moment, except for such trivial 
cases, we can not 
calculate $b_{f,m}^{{\rm mero}}(s)$ explicity. 
Instead, by \cite[Theorem 3.3 and Corollary 3.4]{N-T} 
for many $f= \frac{F}{G}$ we can calculate 
${\rm E}_{f,x_0}$ completely. Namely, for 
$m \geq 2 \dim X$ the roots of 
$b_{f,m}^{{\rm mero}}(s)$ of such $f$ 
can be determined up to some shifts of 
integers and multiplicities. 
Moreover in Section \ref{sec:s4}, 
we also give an upper bound 
\begin{equation}
(b_{f,m}^{{\rm mero}})^{-1}(0) \subset 
B_{f,m}^{\pi} \subset \QQ \qquad (m \geq 0)
\end{equation}
for the roots of $b_{f,m}^{{\rm mero}}(s)$ 
described in terms of 
resolutions of singularities 
$\pi : Y \longrightarrow X$ 
of the divisor $D \subset X$ such that 
$\pi^{-1}(D) \subset Y$ is normal crossing. 
If $G=1$ and $f$ is holomorphic, 
this corresponds to the negativity  
of the roots of b-functions 
proved by Kashiwara \cite{K-1}. 
Indeed, in particular for $m=0$ our upper bound 
means that the roots of $b_{f,0}^{{\rm mero}}(s)$ 
are negative rational numbers. 
Moreover by defining a reduced b-function 
$\tilde{b}_{f}^{{\rm mero}} (s)$ of $f$ 
we obtain also a lower bound 
\begin{equation}
( \tilde{b}_{f}^{{\rm mero}})^{-1}(0)
\subset 
(b_{f,m}^{{\rm mero}})^{-1}(0) \subset \QQ 
\qquad (m \geq 0). 
\end{equation} 
This $\tilde{b}_{f}^{{\rm mero}} (s)$ 
could be a candidate for the b-function 
of the meromorphic function $f$. 
However to our regret, as we shall see in Proposition 
\ref{PROP-1}, it has much less information 
on the singularities 
of $f$ than $b_{f,m}^{{\rm mero}}(s)$. 
See Section \ref{sec:s4} for the details. 
Finally in Section \ref{sec:s5}, we introduce 
multiplier ideal sheaves for the 
meromorphic function $f= \frac{F}{G}$ 
and show that their jumping numbers are 
contained in the set 
\begin{equation}
\bigcup_{i=0,1,2, \ldots}
\Bigl\{ -(b_{f,0}^{{\rm mero}})^{-1}(0)+i \Bigr\}
\subset \QQ_{>0}. 
\end{equation}
This is an analogue for meromorphic functions 
of the main theorem of Ein-Lazarsfeld-Smith-Varolin 
\cite{E-L-S-V}. See Corollary \ref{Cor-jump} 
for the details. 

\par 
After we posted this paper to the arXiv, we were 
informed from the authors {\`A}lvarez Montaner, 
Gonz{\'a}lez Villa, Le{\'o}n-Cardenal and N{\'u\~n}ez-Betancourt 
of \cite{A-G-L-N} that they were also developing a 
theory of b-functions for meromorphic functions 
similar to but different from ours. 
Among other things, for the meromorphic function 
$f= \frac{F}{G}$ they define their b-function 
$b_{F/G}(s) \in \CC [s]$ to be the minimal 
polynomial $b(s) \not= 0$ satisfying the equation 
\begin{equation}
b(s) f^{s} \in \D_X[s] f^{s+1} 
\end{equation}
and apply it to the studies of the analytic 
continuations of Archimedian local zeta functions 
and multiplier ideals associated to $f= \frac{F}{G}$. 
Moreover, in \cite[Theorem 6.7]{A-G-L-N} 
they obtain a result on the jumping numbers of 
multiplier ideals similar to Corollary \ref{Cor-jump}. 
Since our $b_{f,0}^{{\rm mero}}(s)$ divides 
their $b_{F/G}(s)$, it is not clear if 
Corollary \ref{Cor-jump} follows 
from \cite[Theorem 6.7]{A-G-L-N}. 
In addition, our b-function $b_{f,0}^{{\rm mero}}(s)$ 
satisfies a nice relationship with the V-filtration of 
a holonomic D-module (see Theorem 
\ref{Main-5}). From this, we see also that  
the minimal jumping number $\alpha >0$ is 
equal to the negative of the largest root 
of $b_{f,0}^{{\rm mero}}(s)$ 
(see Corollary \ref{Cor-jump}). Altogether, 
the results in \cite{A-G-L-N} look very useful 
and complementary to ours. 
Especially for some basic properties of 
the multiplier ideals, we refer to 
\cite[Sections 6 and 7]{A-G-L-N}.

\bigskip
\noindent{\bf Acknowledgement:} 
The author thanks Professors N{\'u\~n}ez-Betancourt, 
Kakehi, Oaku, Sabbah and Schapira 
for some useful discussions during the 
preparation of this paper. He is also 
grateful to the anonymous referee 
whose comments and suggestions were very helpful.

\section{Proof of Theorem \ref{Main-1}}\label{sec:s2}

We follow the classical arguments of 
Kashiwara \cite{K-1}, Malgrange \cite{M}, 
Gyoja \cite{G} and Sabbah \cite{M-S}. 
For the theory of D-modules, we refer to 
\cite{Dimca}, \cite{H-T-T}, \cite{K-book}, 
\cite{K-S}, \cite{M-S} and use freely 
the notions and the terminologies in them. 
Let $\CC [s,t]$ be the $\CC$-algebra generated by 
the two elements $s,t$ satisfying the relation 
$ts=(s+1)t$ i.e. $[t,s]=t$. Similarly, 
we define $\CC [s, t^{\pm}]$, 
$\D_X [s,t]$ and $\D_X [s, t^{\pm}]$. Then there 
exists a natural isomorphism 
\begin{equation}
\CC [s, t^{\pm}] \simto \CC [t, \partial_t]
\Bigl[ \frac{1}{t} \Bigr] 
\qquad (s \longmapsto - \partial_t t) 
\end{equation}
of $\CC$-algebras (see Gyoja \cite{G}) and the one 
\begin{equation}
\D_X [s, t^{\pm}] \simto 
 ( \D_X \otimes_{\CC_X} \CC_X[t, \partial_t]) 
\Bigl[ \frac{1}{t} \Bigr] 
\end{equation}
of $\D_X$-algebras associated to it. In the product 
space $X \times \CC$ we define a hypersurface 
$Z \subset X \times \CC$ by 
\begin{equation}
Z= \{ (x,t) \in X \times \CC \ | \ 
tG(x)-F(x)=0 \}. 
\end{equation}
Note that $Z$ is the closure of the graph of 
the meromorphic function $f= \frac{F}{G}: 
X \setminus G^{-1}(0) \longrightarrow \CC$ 
in $X \times \CC$. Let 
\begin{equation}
\HH^1_{[Z]}( \sho_{X \times \CC}) \simeq 
\frac{ \sho_{X \times \CC} [ \frac{1}{tG-F} ]}{
\sho_{X \times \CC} }
\end{equation}
be the first local cohomology sheaf of 
$\sho_{X \times \CC}$ along $Z \subset X \times \CC$ 
and define a regular holonomic $\D_{X \times \CC}$-module 
$\M$ by 
\begin{equation}\label{EQ-2} 
\M : = 
\{ \HH^1_{[Z]}( \sho_{X \times \CC}) \} 
\Bigl[ \frac{1}{G} \Bigr] 
 \simeq \frac{\sho_{X \times \CC} [ \frac{1}{(tG-F)G} ]
}{\sho_{X \times \CC} [ \frac{1}{G} ]}
\end{equation}
which is endowed with the canonical section 
\begin{equation}
\delta (t-f(x)):= 
 \Bigl[ \frac{1}{t-f(x)} \Bigr]= 
 \Bigl[ \frac{G(x)}{tG(x)-F(x)} \Bigr] \in \M . 
\end{equation}
Unlike the classical case where $f$ is holomorphic, 
this section does not necessarily generate $\M$ 
over $\D_{X \times \CC}$ (see Lemma \ref{Lemma-3} below). 
Nevertheless, 
as in Malgrange \cite{M} and Gyoja \cite{G} 
for any non-negative integer $m \geq 0$ 
there exists an isomorphism 
\begin{equation}\label{eq-fun}
\D_X [s, t^{\pm}] \Bigl( \frac{1}{G^m} f^{s} \Bigr)
\simto 
 ( \D_X \otimes_{\CC_X} \CC_X[t, \partial_t]) 
\Bigl[ \frac{1}{t} \Bigr] 
\Bigl( \frac{1}{G^m} \delta (t-f(x)) \Bigr) 
\end{equation}
($\frac{1}{G^m} f^{s} \longmapsto \frac{1}{G^m} \delta (t-f(x))$) 
on a neighborhood of $F^{-1}(0) \subset X$ 
which is linear over 
$\D_X [s, t^{\pm}] \simeq 
 ( \D_X \otimes_{\CC_X} \CC_X[t, \partial_t]) 
[ \frac{1}{t}]$. 
Since there is no non-zero section of $\M$ 
supported in $G^{-1}(0) \times \CC \subset 
X \times \CC$ by Hilbert's nullstellensatz, 
to show \eqref{eq-fun} it suffices to 
compare the annihilators of the 
generators of its both sides on 
$X \setminus G^{-1}(0)$. 
Here the right hand side of \eqref{eq-fun} is 
understood to be a subsheaf of 
$( \M |_{ \{ t=0 \} } ) [ \frac{1}{t}]$ and 
the multiplication by $t$ on it corresponds to 
the action $s \longmapsto s+1$ on the left hand 
side (see e.g. Gyoja \cite{G} for the details). 
Restricting the isomorphism  \eqref{eq-fun} 
to a subsheaf, we obtain an isomophism 
\begin{equation}
\D_X [s] \Bigl( \frac{1}{G^m} f^{s} \Bigr)
\simto 
\D_X [- \partial_t t] 
\Bigl( \frac{1}{G^m} \delta (t-f(x)) \Bigr). 
\end{equation}
Now let us consider the $V$-filtration 
$\{ V_j( \D_{X \times \CC} ) \}_{j \in \ZZ}$ 
of $\D_{X \times \CC}$ along the hypersurface 
$\{ t=0 \} = X \times \{ 0 \} \subset 
X \times \CC$. Similarly, we define 
a filtration 
$\{ V_j( \D_X \otimes_{\CC_X} \CC_X[t, \partial_t] ) 
\}_{j \in \ZZ}$ 
of $\D_X \otimes_{\CC_X} \CC_X[t, \partial_t] 
\subset \D_{X \times \CC} |_{ \{ t=0 \} }$. 
Denote the section 
\begin{equation}
\frac{1}{G^m} \delta (t-f(x)) \in \M |_{ \{ t=0 \} }
\end{equation}
of $\M |_{ \{ t=0 \} }$ simply by $\sigma_m$. 
Then by $t \cdot \delta (t-f)=f \cdot \delta (t-f)$ 
we obtain isomorphisms 
\begin{align}
V_0( \D_X \otimes_{\CC_X} \CC_X[t, \partial_t] ) \sigma_m 
 & \simeq 
\sum_{k=0}^{+ \infty} \D_X [s]
\Bigl( \frac{1}{G^m} f^{s+k} \Bigr), 
\\
V_{-1}( \D_X \otimes_{\CC_X} \CC_X[t, \partial_t] ) \sigma_m 
 & \simeq 
\sum_{k=1}^{+ \infty} \D_X [s]
\Bigl( \frac{1}{G^m} f^{s+k} \Bigr).
\end{align}
This implies that the $V_0( \D_X \otimes_{\CC_X} 
\CC_X[t, \partial_t] )$-module 
\begin{equation}
\K := \frac{V_{0}( \D_X \otimes_{\CC_X} \CC_X[t, \partial_t] ) 
\sigma_m }{V_{-1}( \D_X \otimes_{\CC_X} \CC_X[t, \partial_t] ) \sigma_m }
\end{equation}
is isomorphic to 
\begin{equation}\label{EQ-1}
\frac{  \sum_{k=0}^{+ \infty} \D_X [s]
( \frac{1}{G^m} f^{s+k} ) }{ 
\sum_{k=1}^{+ \infty} \D_X [s]
( \frac{1}{G^m} f^{s+k} ) }. 
\end{equation}
Here we 
used the identification 
\begin{equation}
V_{0}( \D_X \otimes_{\CC_X} \CC_X[t, \partial_t] ) 
\simeq \D_X [s,t]
\end{equation}
given by $- \partial_t t \longmapsto s$. 
Moreover by Lemma \ref{Lemma-1} below there exists 
also an isomorphism 
\begin{equation}
( \sho_{X \times \CC} |_{ \{ t=0 \} } ) 
\otimes_{\sho_X \otimes_{\CC_X} \CC_X[t]} 
\K \simeq \K^{\infty}  := 
\frac{ V_0( \D_{X \times \CC} )\sigma_m 
}{V_{-1}( \D_{X \times \CC} )\sigma_m}. 
\end{equation}
By the classical result on the specializability of 
$\M$ along $\{ t=0 \}$, there exists a non-zero 
polynomial $b(s) \in \CC [s]$ such that 
\begin{equation}
b( - \partial_t t ) \sigma_m \in 
V_{-1}( \D_{X \times \CC} ) \sigma_m.
\end{equation}
This condition is equivalent to the one that 
the image 
\begin{equation}\label{eq-im} 
\G := {\rm Im} [ b( - \partial_t t ) : 
\K^{\infty} \longrightarrow \K^{\infty} ]
\end{equation}
is zero. Note that the sheaf homomorphism 
\begin{equation}
b( - \partial_t t ) : \K \longrightarrow \K
\end{equation}
is $\sho_X \otimes_{\CC_X} \CC_X[t]$-linear and 
the above one in \eqref{eq-im} 
is obtained by applying 
the tensor product 
$( \sho_{X \times \CC} |_{ \{ t=0 \} } ) 
\otimes_{\sho_X \otimes_{\CC_X} \CC_X[t]} 
( \cdot )$ to it. 
Since $\sho_{X \times \CC} |_{ \{ t=0 \} } $ is flat over 
$\sho_X \otimes_{\CC_X} \CC_X[t]$, we thus obtain an 
isomorphism 
\begin{equation}
\G \simeq 
( \sho_{X \times \CC} |_{ \{ t=0 \} } ) 
\otimes_{\sho_X \otimes_{\CC_X} \CC_X[t]} 
{\rm Im} [b( - \partial_t t ) : \K \longrightarrow \K ]. 
\end{equation}
By $\G \simeq 0$ and the faithfully flatness of 
$\sho_{X \times \CC} |_{ \{ t=0 \} } $ over 
$\sho_X \otimes_{\CC_X} \CC_X[t]$, we obtain also 
\begin{equation}
{\rm Im} [b( - \partial_t t ) : \K \longrightarrow \K ]
\simeq 0. 
\end{equation}
It follows from the previous description 
\eqref{EQ-1} of $\K$ 
that we have the desired condition 
\begin{equation}
b(s) \Bigl( \frac{1}{G^m} f^{s} \Bigr) \in 
\sum_{k=1}^{+ \infty} \D_X [s]
\Bigl( \frac{1}{G^m} f^{s+k} \Bigr). 
\end{equation}
This completes the proof. 
\fin 

\medskip 
\par
\begin{lemma}\label{Lemma-1} 
There exists an isomorphism 
$( \sho_{X \times \CC} |_{ \{ t=0 \} } ) 
\otimes_{\sho_X \otimes_{\CC_X} \CC_X[t]} 
\K \simeq \K^{\infty}$. 
\end{lemma}

\begin{proof} 
By our construction of the regular holonomic 
$\D_{X \times \CC}$-module $\M$ in the proof 
of Theorem \ref{Main-1}, there exist a natural 
morphism 
\begin{equation}
\Phi : \M_0 := 
\frac{( \sho_X \otimes_{\CC_X} \CC_X[t] ) [ \frac{1}{(tG-F)G} ]
}{( \sho_X \otimes_{\CC_X} \CC_X[t] ) [ \frac{1}{G} ]} 
\longrightarrow \M |_{ \{ t=0 \} } 
\end{equation}
of $\sho_X \otimes_{\CC_X} \CC_X[t]$-modules. 
Since $F, G \in \sho_X$ are coprime each 
other, the same is true also for 
$tG-F, G \in \sho_{X \times \CC}$ and hence the 
morphism $\Phi$ is injective. Therefore, 
for $j=0,-1$ the 
$V_{0}( \D_X \otimes_{\CC_X} \CC_X[t, \partial_t] )$-module 
$V_{j}( \D_X \otimes_{\CC_X} \CC_X[t, \partial_t] ) \sigma_m 
\subset \M |_{ \{ t=0 \} }$ is isomorphic to the 
image of the morphism 
\begin{equation}
V_{j}( \D_X \otimes_{\CC_X} \CC_X[t, \partial_t] ) 
\longrightarrow \M_0 \qquad (P \longmapsto P \sigma_m ). 
\end{equation}
By the ismorphisms 
\begin{equation}
( \sho_{X \times \CC} |_{ \{ t=0 \} } ) 
\otimes_{\sho_X \otimes_{\CC_X} \CC_X[t]} 
V_{j}( \D_X \otimes_{\CC_X} \CC_X[t, \partial_t] )
\simeq V_j( \D_{X \times \CC} ) |_{ \{ t=0 \} }  \qquad (j=0,-1), 
\end{equation}
\begin{equation}
( \sho_{X \times \CC} |_{ \{ t=0 \} } ) 
\otimes_{\sho_X \otimes_{\CC_X} \CC_X[t]} 
\M_0 \simeq \M |_{ \{ t=0 \} }
\end{equation}
and the 
flatness of $\sho_{X \times \CC} |_{ \{ t=0 \} } $ over 
$\sho_X \otimes_{\CC_X} \CC_X[t]$, we obtain 
isomorphisms 
\begin{align*}
V_j( \D_{X \times \CC} )  \sigma_m 
=  & {\rm Im}[ 
V_j( \D_{X \times \CC} )|_{ \{ t=0 \} }
\longrightarrow  \M |_{ \{ t=0 \} } ]
\\
\simeq & 
( \sho_{X \times \CC} |_{ \{ t=0 \} } ) 
\otimes_{\sho_X \otimes_{\CC_X} \CC_X[t]}
{\rm Im}[ 
V_j( \D_X \otimes_{\CC_X} \CC_X[t, \partial_t] ) 
\longrightarrow  \M_0 ]
\\
\simeq & 
( \sho_{X \times \CC} |_{ \{ t=0 \} } )
\otimes_{\sho_X \otimes_{\CC_X} \CC_X[t]}
V_j( \D_X \otimes_{\CC_X} \CC_X[t, \partial_t] ) \sigma_m 
\qquad (j=0,-1). 
\end{align*}
Then the assertion immediately follows. 
\end{proof}

\section{Proof of Theorem \ref{Main-2}}\label{sec:s3}

First of all, we shall recall the classical theory 
of Kashiwara-Malgrange filtrations. For more precise explanations on 
them, we refer to \cite{K-2} and \cite{M-S}. 
We assume first that $\M$ is a general regular holonomic 
$\D_{X \times \CC}$-module on the 
product of a complex manifold $X$ and $\CC_t$. 
Set $\theta =t \partial_t \in \D_{X \times \CC}$ 
and for a section $\sigma \in \M$ of $\M$ denote 
by $p_{\sigma}(s) \in \CC [s]$ the minimal 
polynomial such that 
\begin{equation}
p_{\sigma}( \theta ) \sigma \in 
V_{-1}( \D_{X \times \CC} ) \sigma. 
\end{equation}
Furthermore, we set 
\begin{equation}
{\rm ord}_{ \{ t=0 \} } ( \sigma ):=
p_{\sigma}^{-1}( 0 ) \subset \CC. 
\end{equation}
On the set $\CC$ of complex numbers let us consider 
the lexicographic order $\geq$ defined by 
\begin{equation}
z \geq w \ \Longleftrightarrow \ 
{\rm Re} z > {\rm Re} w \quad \text{or} \quad 
{\rm Re} z = {\rm Re} w, \ {\rm Im} z \geq {\rm Im} w.
\end{equation}
Then for $\alpha \in \CC$ we define a 
$V_0( \D_{X \times \CC} )$-submodule 
$V_{\alpha}( \M )$ of $\M$ by 
\begin{equation}
V_{\alpha}( \M )= \{ \sigma \in \M \ | \ 
{\rm ord}_{ \{ t=0 \} } ( \sigma ) \geq 
- \alpha -1 \}. 
\end{equation}
We can easily see that there exists a 
finite subset $A \subset \{ z \in \CC \ | \ 
-1 \leq z <0 \}$ such that for any section $\sigma \in \M$ of $\M$ 
we have 
\begin{equation}
{\rm ord}_{ \{ t=0 \} } ( \sigma ) \subset A+ \ZZ.
\end{equation}
Moreover for each element $\alpha \in A$ of such $A$ 
the filtration $\{ V_{\alpha +j}( \M ) \}_{j \in \ZZ}$ 
of $\M$ is a good $V$-filtration. 
For  $\alpha \in A+ \ZZ$ we set 
\begin{equation}
V_{< \alpha}( \M )= \bigcup_{\beta < \alpha} V_{\beta}( \M )
= \{ \sigma \in \M \ | \ 
{\rm ord}_{ \{ t=0 \} } ( \sigma ) >- \alpha -1 \}
\end{equation}
and 
\begin{equation}
{\rm gr}^V_{\alpha}( \M )= 
V_{\alpha}( \M )/ V_{< \alpha}( \M ). 
\end{equation}
Then ${\rm gr}^V_{\alpha}( \M )$ is a regular holonomic 
$\D_X$-module and we can easily show that there 
exists $N \gg 0$ such that 
\begin{equation}
( \theta + \alpha +1)^N {\rm gr}^V_{\alpha}( \M ) =0. 
\end{equation}
The following lemma is well-known to the specialists. 

\begin{lemma}\label{Lemma-2} 
Let $\sigma \in \M$ be a section of $\M$ such 
that $\D_{X \times \CC} \sigma = \M$. Then 
\begin{enumerate}
\item[\rm{(i)}] 
For any section $\tau \in \M$ of $\M$ 
we have 
\begin{equation}
{\rm ord}_{ \{ t=0 \} } ( \tau ) \subset 
{\rm ord}_{ \{ t=0 \} } ( \sigma ) + \ZZ. 
\end{equation}
\item[\rm{(ii)}] 
For any $\lambda \in {\rm ord}_{ \{ t=0 \} } ( \sigma )$ 
we have 
\begin{equation}
{\rm gr}^V_{- \lambda -1}( \M ) \not= 0. 
\end{equation}
\item[\rm{(iii)}] 
Conversely, if ${\rm gr}^V_{\alpha}( \M ) \not= 0$, 
then we have 
\begin{equation}
- \alpha -1 \in 
{\rm ord}_{ \{ t=0 \} } ( \sigma ) + \ZZ. 
\end{equation}
\end{enumerate}
\end{lemma}

Now we return to the situation in the proof of 
Theorem \ref{Main-1}. Namely, for the meromorphic 
function $f= \frac{F}{G}$ we have 
\begin{equation}
\M  \simeq \frac{\sho_{X \times \CC} [ \frac{1}{(tG-F)G} ]
}{\sho_{X \times \CC} [ \frac{1}{G} ]} 
\end{equation}
and 
\begin{equation}
\sigma_m = \frac{1}{G^m} \delta (t-f(x)) 
= \Bigl[ \frac{G}{(tG-F)G^m} \Bigr] \in \M \qquad 
(m \geq 0). 
\end{equation}
Then we have the following result, whose proof 
is inspired from Sabbah's exposition \cite{Sabbah-3}. 

\begin{lemma}\label{Lemma-3} 
Assume that $m \geq 2 \dim X$. Then $\M$ 
is generated by the section $\sigma_m \in \M$ 
over $\D_{X \times \CC}$ i.e. 
$\M = \D_{X \times \CC} \sigma_m$. 
\end{lemma}

\begin{proof} 
Set $g:=(tG(x)-F(x)) \cdot G(x) \in \sho_{X \times \CC}$ 
and let $b_g(s) \in \CC [s]$ be its Bernstein-Sato 
polynomial. Then by \cite[Theorem 0.4]{S}, 
for any root $\alpha \in \QQ$ of $b_g(s)$ we have 
\begin{equation}
- \dim (X \times \CC)= - \dim X -1 < \alpha <0. 
\end{equation}
Moreover for any $k \geq 1$ there exists 
$P_k(s) \in \D_{X \times \CC} [s]$ such that 
\begin{equation}
b_g(s-k) \cdots \cdots b_g(s-2) b_g(s-1) g^{s-k} 
= P_k(s) g^s. 
\end{equation}
Set $n:= \dim X$. Then by substituting $s$ 
in the above formula by $-n$ we see that 
for any $k \geq 1$ the meromorphic function 
$g^{-n-k}$ is a non-zero constant multiple of 
$P_k(-n) g^{-n}$. This implies that 
\begin{equation}
\M  \simeq \frac{\sho_{X \times \CC} [ \frac{1}{(tG-F)G} ]
}{\sho_{X \times \CC} [ \frac{1}{G} ]} 
\end{equation}
is generated by its section 
\begin{equation}
\Bigl[ \frac{1}{g^n} \Bigr] = 
\Bigl[ \frac{1}{(tG-F)^nG^n} \Bigr] \in \M 
\end{equation}
over $\D_{X \times \CC}$. On the other hand, 
the section $\partial_t^{n-1} \sigma_m \in \M$ of $\M$ 
is a non-zero constant multiple of 
\begin{equation}
\Bigl[ \frac{1}{(tG-F)^{n}G^{m-n}} \Bigr] \in \M. 
\end{equation}
Therefore, if $m \geq 2n= 2 \dim X$, it 
generates $\M$ over $\D_{X \times \CC}$. 
\end{proof}

Now let us prove Theorem \ref{Main-2}. 
By the proof of Theorem \ref{Main-1} 
and the correspondence 
$s \longleftrightarrow - \partial_t t = - \theta -1$ 
the Bernstein-Sato polynomial 
$b_{f,m}^{{\rm mero}}(s)$ of $f$ coincides with 
$p_{\sigma_m}(-s-1)$. This in particular implies that 
we have 
\begin{equation}
(b_{f,m}^{{\rm mero}})^{-1}(0)= 
\{ - \lambda -1 \ | \ \lambda \in 
{\rm ord}_{ \{ t=0 \} } ( \sigma_m ) \}. 
\end{equation}
Note also that for the $\D_X$-module $\M$ in 
the proof of Theorem \ref{Main-1} we have 
an isomorphism 
\begin{equation}
{\rm DR}_{X \times \CC} ( \M ) \simeq 
R \Gamma_{(X \setminus G^{-1}(0)) \times \CC} 
( \CC_Z ) [n]
\end{equation}
and the nearby cycle sheaf 
$\psi_t({\rm DR}_{X \times \CC} ( \M ) )$ coincides 
with the meromorphic nearby cycle 
$\psi_f^{{\rm mero}} ( \CC_X )$ 
introduced in \cite{N-T} up to some shift. 
Assume first that $m \geq 2 \dim X$. Then 
by Lemma \ref{Lemma-3} the section 
$\sigma_m \in \M$ generates $\M$ over $\D_{X \times \CC}$ 
and the second assertion of Theorem \ref{Main-2} 
follows from Lemma \ref{Lemma-2}, 
Kashiwara's isomorphism 
\begin{equation}
\bigoplus_{-1 \leq \alpha <0} 
{\rm DR}_X ( {\rm gr}^V_{\alpha}( \M ) ) 
\simeq 
\psi_t( {\rm DR}_{X \times \CC} ( \M ))
\simeq \bigoplus_{-1 \leq \alpha <0} 
\psi_{t, \exp (2 \pi i \alpha )}
( {\rm DR}_{X \times \CC} ( \M ))
\end{equation}
and \cite[Lemma 2.1 (iii)]{N-T}. 
If we do not have the condition $m \geq 2 \dim X$, 
by considering the $\D_{X \times \CC}$-submodule 
$\D_{X \times \CC} \sigma_m \subset \M$ 
instead of $\M$ itself we obtain the first 
assertion of Theorem \ref{Main-2}.  
This completes the proof. 
\fin 

\medskip 
\par
By the proofs of Theorems \ref{Main-1} and 
\ref{Main-2} we obtain 
the following weak relation among the roots of 
the b-functions $b_{f,m}^{{\rm mero}}(s)$ for various $m \geq 0$. 

\begin{lemma}\label{Lemma-4} 
Let $m, m^{\prime} \geq 0$ be two non-negative integers 
such that $m \geq m^{\prime}$. Then for some $l \gg 0$ 
we have an inclusion 
\begin{equation}
(b_{f,m^{\prime}}^{{\rm mero}})^{-1}(0) \subset 
\bigcup_{i=0}^l \Bigl\{ (b_{f,m}^{{\rm mero}})^{-1}(0)-i 
\Bigr\}. 
\end{equation}
\end{lemma}

\begin{proof} 
By the proofs of Theorems \ref{Main-1} and \ref{Main-2} we have 
$b_{f,m}^{{\rm mero}}(s)=p_{\sigma_m}(-s-1)$ and 
$b_{f,m^{\prime}}^{{\rm mero}}(s)=
p_{\sigma_{m^{\prime}}}(-s-1)$. Moreover by our 
assumption $m \geq m^{\prime}$ we have 
\begin{equation}
\sigma_{m^{\prime}} = G^{m-m^{\prime}} \cdot 
\sigma_m \in \sho_X \sigma_m 
\subset V_0( \D_{X \times \CC} ) \sigma_m. 
\end{equation}
Set $\N := \D_{X \times \CC} \sigma_m$ and 
$\N^{\prime} := \D_{X \times \CC} \sigma_{m^{\prime}}$. 
Then the $V$-filtration 
$\{ V_j( \D_{X \times \CC} ) \sigma_m \}_{j \in \ZZ}$ 
(resp. $\{ V_j( \D_{X \times \CC} ) 
\sigma_{m^{\prime}} \}_{j \in \ZZ}$) 
of $\N$ (resp. $\N^{\prime}$) is good. By Artin-Rees's lemma, 
the $V$-filtration $\{ U_j( \N^{\prime} ) \}_{j \in \ZZ}$ 
of $\N^{\prime}$ defined by 
\begin{equation}
U_j( \N^{\prime} ):= \N^{\prime} \cap 
\Bigl( V_j( \D_{X \times \CC} ) \sigma_m 
\Bigr) \qquad (j \in \ZZ )
\end{equation}
is also good and satisfies the condition 
$\sigma_{m^{\prime}} \in U_0( \N^{\prime} )$. 
Then there exists $l \gg 0$ such that 
\begin{equation}
U_{-l-1}( \N^{\prime} ) \subset V_{-1}( \D_{X \times \CC} ) 
\sigma_{m^{\prime}} \subset \N^{\prime}. 
\end{equation}
This implies that we have 
\begin{equation}
p_{\sigma_m}( \theta -l) \cdots p_{\sigma_m}( \theta -1) 
p_{\sigma_m}( \theta ) \sigma_{m^{\prime}}
\in V_{-1}( \D_{X \times \CC} ) 
\sigma_{m^{\prime}}.
\end{equation}
Then the assertion immediately follows. 
\end{proof}

\section{Upper and lower bounds for the roots of 
b-functions}\label{sec:s4}

 Recall that in \cite{K-1} Kashiwara proved that 
if $f$ is holomorphic the roots of 
the Bernstein-Sato polynomial 
$b_f(s)$ are negative rational numbers. 
In this section, we prove an analogous result 
for the meromorphic function $f= \frac{F}{G}$. 
We can easily prove 
that the roots of our b-function 
$b_{f,m}^{{\rm mero}}(s)$ are rational numbers, 
but their negativity does not follow from 
our proof. For this reason, here we only give an upper bound 
\begin{equation}
(b_{f,m}^{{\rm mero}})^{-1}(0) \subset 
B_{f,m}^{\pi} \subset \QQ \qquad (m \geq 0)
\end{equation}
for the set $(b_{f,m}^{{\rm mero}})^{-1}(0)$ 
in terms of resolutions of singularities $\pi$ 
of $D \subset X$. 
The precise statement is as follows. 
Let $\pi : Y \longrightarrow X$ be a 
resolution of singularities 
of the divisor $D=
F^{-1}(0) \cup G^{-1}(0) \subset X$, which means that 
$\pi : Y \longrightarrow X$ is a proper morphism 
of $n$-dimensional 
complex manifolds such that  
$\pi^{-1}(D) \subset Y$ is normal crossing 
and $\pi |_{Y \setminus \pi^{-1}(D)}: 
Y \setminus \pi^{-1}(D) \longrightarrow X 
\setminus D$ is an isomorphism. 
Then we define a meromorphic function $g$ on $Y$ by 
\begin{equation}
g:=f \circ \pi = \frac{F \circ \pi}{G \circ \pi}.
\end{equation}
From now on, we fix a non-negative integer $m \geq 0$ 
and consider the (local) Bernstein-Sato polynomials 
of $g$ of order $m$. At each point $q \in \pi^{-1}(D)$ of 
the normal crossing divisor $\pi^{-1}(D)$, 
there exists a local coordinate system 
$y=(y_1,y_2, \ldots, y_n)$ such that 
$q=(0,0, \ldots , 0)$ and 
\begin{equation}
(F \circ \pi )(y)  =  \prod_{i=1}^n y_i^{a_i} 
\quad (a_i \geq 0), \qquad 
(G \circ \pi )(y)  =  \prod_{i=1}^n y_i^{b_i} 
\quad (b_i \geq 0). 
\end{equation}
Then we have 
\begin{equation}
\bigl( \frac{1}{G^m} f^{s} \Bigr) (y) = 
\prod_{i=1}^n y_i^{(a_i-b_i)s-m b_i}. 
\end{equation}
It follows that the set $K_q \subset \QQ$ 
of the roots of the (local) Bernstein-Sato polynomial 
of $g$ at $q$ is explicitly given by 
\begin{equation}
K_q= \bigcup_{i: \ a_i > b_i} 
\Bigl\{ \frac{m b_i}{a_i-b_i}
 - \frac{k}{a_i-b_i} \ | \ 
1 \leq k \leq a_i-b_i \Bigr\} \subset \QQ
\end{equation}
(see e.g. \cite[Lemma 6.10]{K-book}). 
It is clear this set $K_q$ does not depend 
on the choice of the local coordinates. 
For the point $x_0 \in D$ its inverse 
image $\pi^{-1}(x_0) \subset \pi^{-1}(D)$ 
being compact, we obtain a finite subset 
\begin{equation}
K:= \bigcup_{q \in \pi^{-1}(x_0)} K_q \subset \QQ. 
\end{equation}

\begin{theorem}\label{Main-3} 
For any $m \geq 0$ 
the roots of the (local)  Bernstein-Sato polynomial 
$b_{f,m}^{{\rm mero}}(s)$ of $f$ at $x_0 \in D$ 
are contained in the set 
\begin{equation}
B_{f,m}^{\pi} := 
\bigcup_{l=0,1,2, \ldots} \bigl( K-l \bigr) =
\{ r-l \ | \ r \in K, \ l=0,1,2, \ldots \} \subset \QQ. 
\end{equation}
In particular, for $m=0$ 
the roots of $b_{f,0}^{{\rm mero}}(s)$ 
are negative rational numbers. 
\end{theorem}

\begin{proof}
Our proof is similar to the one in 
Kashiwara \cite{K-1} for the case where 
$f$ is holomorphic. But we also need some 
new ideas to treat the meromorphic case. 
Recall that for the section $\sigma := \sigma_m \in \M$ 
(see \eqref{EQ-2}) of the 
regular holonomic $\D_{X \times \CC}$-module 
$\M$ we denote by $p_{\sigma}(s) \in \CC [s]$ the 
minimal polynomial $p(s) \not= 0$ such that 
\begin{equation}
p( \theta ) \sigma \in V_{-1}( \D_{X \times \CC} ) 
\sigma
\end{equation}
and we have $b_{f,m}^{{\rm mero}}(s) = 
p_{\sigma}(-s-1)$. Let $i: Y \longrightarrow Y \times X$ 
($y \longmapsto (y, \pi (y))$) be the graph 
embedding by $\pi$ and $p: Y \times X \rightarrow X$ 
($(y, x) \longmapsto x$) the projection such 
that $\pi = p \circ i$. We set also 
\begin{align*}
\tilde{i} :=i \times {\rm id}_{\CC} 
: Y \times \CC \longrightarrow (Y \times X) \times \CC, 
\\
\tilde{p} :=p \times {\rm id}_{\CC} 
: (Y \times X) \times \CC \longrightarrow X \times \CC
\end{align*}
so that we have $\tilde{\pi} := \pi \times {\rm id}_{\CC} 
= \tilde{p} \circ \tilde{i}$. As in the case 
of the meromorphic function $f= \frac{F}{G}$, 
we define a regular holonomic $\D_{Y \times \CC}$-module 
$\N$ associated to $g= \frac{F \circ \pi}{G \circ \pi}$ 
and its section 
\begin{equation}
\tau := \frac{1}{\{ (G \circ \pi )(y) \}^m} \delta 
(t-g(y)) \in \N. 
\end{equation}
Then the roots of its minimal 
polynomial $p_{\tau}(s) \in \CC [s]$ such that 
\begin{equation}
p_{\tau}( \theta ) \tau \in 
V_{-1}( \D_{Y \times \CC} ) \tau
\end{equation}
is contained in the set 
$\{ -r-1 \ | \ r \in K \} \subset \QQ$. 
Since $\tilde{\pi} : Y \times \CC \longrightarrow 
X \times \CC$ is an isomorphism over 
$(Y \setminus \pi^{-1}(D)) \times \CC \simeq 
(X \setminus D) \times \CC$, the section $\tau \in \N$ 
is naturally identified with $\sigma \in \M$ there. Let 
\begin{equation}
\bfD \tilde{i}_* \N \simeq H^0 \bfD \tilde{i}_* \N 
= \tilde{i}_* \Bigl( \D_{(Y \times X) \times \CC \leftarrow 
Y \times \CC} \otimes_{\D_{Y \times \CC}} \N  \Bigr) 
\end{equation}
be the direct image of $\N$ by $\tilde{i}$ and 
\begin{equation}
\widetilde{\tau} := 
1_{(Y \times X) \times \CC \leftarrow Y \times \CC}
\otimes \tau \in \bfD \tilde{i}_* \N 
\end{equation}
its section defined by $\tau \in \N$. 
Then it is easy to see that the minimal 
polynomial $p_{\widetilde{\tau}}(s) \in \CC [s]$ such that 
\begin{equation}
p_{\widetilde{\tau}}( \theta ) \widetilde{\tau} \in 
V_{-1}( \D_{(Y \times X) \times \CC} ) \widetilde{\tau}
\end{equation}
is equal to $p_{\tau}(s)$. 
Let us consider the $\D_{Y \times \CC}$-submodule 
$\N_0:= \D_{Y \times \CC} \tau \subset \N$ of $\N$ 
generated by $\tau \in \N$. Then we have $\widetilde{\tau}
 \in \bfD \tilde{i}_* \N_0 \subset \bfD \tilde{i}_* \N$ and 
$\D_{(Y \times X) \times \CC} \widetilde{\tau} 
= \bfD \tilde{i}_* \N_0$. Hence we can define a 
good $V$-filtration $\{ U_j( \bfD \tilde{i}_* \N_0 ) 
\}_{j \in \ZZ}$ 
of $\bfD \tilde{i}_* \N_0$ by 
\begin{equation}
 U_j( \bfD \tilde{i}_* \N_0 ):= 
V_j( \D_{(Y \times X) \times \CC} ) \widetilde{\tau} 
\qquad (j \in \ZZ ).
\end{equation}
Then it is easy to see that for any $j \in \ZZ$ we have 
\begin{equation}
p_{\tau}( \theta +j)  U_j( \bfD \tilde{i}_* \N_0 ) \subset 
 U_{j-1}( \bfD \tilde{i}_* \N_0 ).  
\end{equation}
Let us consider the 
relationship between $p_{\widetilde{\tau}}(s) = p_{\tau}(s)$ 
and $p_{\sigma}(s)= b_{f,m}^{{\rm mero}}(-s-1)$. 
For this purpose, let 
\begin{equation}
\M^{\prime} := H^0 \bfD \widetilde{\pi}_* \N_0 \simeq 
H^0 \bfD \tilde{p}_* ( \bfD \tilde{i}_* \N_0 )
\end{equation}
be the 0-th direct image of $\bfD \tilde{i}_* \N_0$ by 
$\tilde{p}$ and as in Gyoja \cite[Section 4.2]{G} 
define its section $\sigma^{\prime} \in \M^{\prime}$ 
to be the image of a section $1_{X \leftarrow Y} \otimes \tau 
\in \tilde{p}_* [ \Omega^n_{Y \times X/X} \otimes_{\sho_{Y \times X}}
\bfD \tilde{i}_* \N_0 ]$ by the morphism 
\begin{equation*}
\tilde{p}_* \bigl[ \Omega^n_{Y \times X/X} 
\otimes_{\sho_{Y \times X}} \bfD \tilde{i}_* \N_0 \bigr] 
\longrightarrow H^0 \bfD \tilde{p}_* ( \bfD \tilde{i}_* \N_0 ) 
= \M^{\prime}
\end{equation*} 
For the construction of 
$1_{X \leftarrow Y} \otimes \tau$, see 
\cite[Section 4.2]{G} for the details. 
Note that on the open subset $(X \setminus D) \times \CC$ 
of $X \times \CC$ we have $\M^{\prime} = 
\D_{X \times \CC} \sigma = \M$ and 
$\sigma^{\prime}$ coincides with $\sigma$. 
For $j \in \ZZ$ we denote by $U_j( \M^{\prime} ) \subset 
\M^{\prime}=H^0 \bfD \tilde{p}_* ( \bfD \tilde{i}_* \N_0 )$ 
the image of the natural morphism 
\begin{equation}
H^0 R \tilde{p}_* \Bigl\{ 
{\rm DR}_{Y \times X/X}
\bigl(   U_j( \bfD \tilde{i}_* \N_0 ) \bigr) \Bigr\} 
\longrightarrow \M^{\prime}. 
\end{equation}
Then by the proof of \cite[Th\'eor\`eme 4.8.1 (1)]{M-S}, 
$\{ U_j( \M^{\prime} ) \}_{j \in \ZZ}$ is a 
good $V$-filtration of $\M^{\prime}$ and for any 
$j \in \ZZ$ we have 
\begin{equation}
p_{\tau}( \theta +j)  U_j( \M^{\prime} ) \subset 
 U_{j-1}( \M^{\prime} ).  
\end{equation}
Moreover by our construction, the section 
$\sigma^{\prime} \in \M^{\prime}$ is contained 
in $U_0( \M^{\prime} )$. Let 
\begin{equation}
\M^{\prime \prime} := 
\D_{X \times \CC} \sigma^{\prime}  \subset \M^{\prime} 
\end{equation}
be the $\D_{X \times \CC}$-submodule of $\M^{\prime}$ 
generated by $\sigma^{\prime}$. Then 
by Artin-Rees's lemma, the $V$-filtration 
$\{ U_j( \M^{\prime \prime} ) \}_{j \in \ZZ}$ of 
$\M^{\prime \prime}$ defined by 
\begin{equation}
U_j( \M^{\prime \prime} ) := \M^{\prime \prime} \cap 
U_j( \M^{\prime} ) \qquad (j \in \ZZ )
\end{equation}
is also good and hence there exists $l \gg 0$ such that 
\begin{equation}
U_{-l}( \M^{\prime \prime} ) = 
\M^{\prime \prime} \cap U_{-l}( \M^{\prime} ) 
\subset V_{-1} ( \D_{X \times \CC} ) \sigma^{\prime}. 
\end{equation}
Combining these results together, we get 
\begin{equation}
p_{\tau}( \theta -(l-1)) \cdots p_{\tau}( \theta -1) p_{\tau}( \theta )
\sigma^{\prime}
\in \M^{\prime \prime} \cap U_{-l}( \M^{\prime} ) 
\subset V_{-1} ( \D_{X \times \CC} ) \sigma^{\prime}. 
\end{equation}
This implies that the minimal polynomial 
$p_{\sigma^{\prime}} (s) \in \CC [s]$ 
for the section $\sigma^{\prime} \in \M^{\prime}$ 
divides the product 
\begin{equation}
p_{\tau}( s -(l-1)) \cdots p_{\tau}( s -1) p_{\tau}( s ) 
\in \CC [s]. 
\end{equation}
Now, according to Kashiwara \cite{K-book}, there exists 
an adjunction morphism 
\begin{equation}
\bfD \widetilde{\pi}_* ( \bfD \widetilde{\pi}^* \M ) 
\longrightarrow  \M 
\end{equation}
of $\D_{X \times \CC}$-modules. Since $\bfD \widetilde{\pi}^* \M$ 
is isomorphic to $\N$ (use e.g. the Riemann-Hilbert correspondence) 
and $\N_0 \subset \N$, 
we obtain a morphism 
\begin{equation}
\Psi : \M^{\prime} = H^0 \bfD \widetilde{\pi}_* \N_0  
\longrightarrow  \M
\end{equation}
of $\D_{X \times \CC}$-modules. 
Then the section $\Psi ( \sigma^{\prime} ) \in \M$ 
of $\M$ coincides with $\sigma \in \M$ on the 
open subset $(X \setminus D) \times \CC \subset 
X \times \CC$. Moreover by the isomorphism 
$\M \simeq \D_X [ \partial_t ]$ on 
the open subset 
$(X \setminus G^{-1}(0)) \times \CC \subset 
X \times \CC$, this coincidence can be extended 
to $(X \setminus G^{-1}(0)) \times \CC$. Here 
we used the classical theorem on the unique 
continuation of holomorphic functions. Since we have 
$\M \simeq \M [ \frac{1}{G} ]$, by Hilbert's nullstellensatz 
we get $\Psi ( \sigma^{\prime} ) = \sigma$ on the 
whole $X \times \CC$. This implies that the minimal 
polynomial $p_{\sigma} (s) = b_{f,m}^{{\rm mero}}(-s-1)$ 
divides the one $p_{\sigma^{\prime}} (s)$. Now the 
assertion is clear. This completes the proof. 
\end{proof}

We have seen that the roots of our b-functions 
$b_{f,m}^{{\rm mero}}(s)$ are rational numbers. 
Let $\rho : \QQ \longrightarrow \QQ / \ZZ$ be the 
quotient map. Then Lemma \ref{Lemma-4} means 
that the subset $A_m:= \rho \ \{ (b_{f,m}^{{\rm mero}})^{-1}(0) \}  
\subset \QQ / \ZZ$ increases with respect to 
$m \geq 0$. By Theorem \ref{Main-2} this sequence 
is stationary for $m \geq 2 \dim X$. 

Next we shall give a lower bound for the 
subsets $(b_{f,m}^{{\rm mero}})^{-1}(0) \subset \QQ$. 
In the proof of Theorem \ref{Main-1}, we have seen 
that the minimal polynomial of $s$ acting on the 
$\D_X$-module
\begin{equation}
\K \simeq 
\frac{  \sum_{k=0}^{+ \infty} \D_X [s]
( \frac{1}{G^m} f^{s+k} ) }{ 
\sum_{k=1}^{+ \infty} \D_X [s]
( \frac{1}{G^m} f^{s+k} ) }
\end{equation}
is equal to our b-function $b_{f,m}^{{\rm mero}}(s)$. 
Localizing it along the hypersurface $G^{-1}(0) \subset 
X$ we obtain a new $\D_X$-module 
\begin{equation}
\K \Bigl[ \frac{1}{G} \Bigr] \simeq 
\frac{  \{ \D_X [s]( \frac{1}{G^m} f^{s}) \} 
[ \frac{1}{G} ] }{ 
\{ \D_X [s]( \frac{1}{G^m} f^{s+1}) \} 
[ \frac{1}{G} ]}
\end{equation}
on which $s$ still acts. Obviously, we have 
$b_{f,m}^{{\rm mero}}(s)=0$ on $\K [ \frac{1}{G} ]$. 
By this observation, we obtain the following result. 
We denote the localized ring $\D_X [ \frac{1}{G} ]$ 
simply by $\widetilde{\D}_X$. 

\begin{theorem}\label{Main-4} 
Let $m \geq 0$ be a non-negative integer. Then 
there exists a non-zero polynomial $b(s) \in \CC [s]$ 
satisfying the equation
\begin{equation}\label{rfe-1}
b(s) \Bigl( \frac{1}{G^m} f^{s} \Bigr) = 
\tilde{P} (s) 
\Bigl( \frac{1}{G^m} f^{s+1} \Bigr)
\end{equation}
for some $\tilde{P} (s) \in \widetilde{\D}_X [s]$. 
\end{theorem}

\begin{definition}\label{RBS-fun}
For $m \geq 0$ we denote by $\tilde{b}_{f,m}^{{\rm mero}}(s) 
\in \CC [s]$ the minimal polynomial 
satisfying the equation in Theorem \ref{Main-4} 
and call it the reduced Bernstein-Sato polynomial 
or the reduced b-function of $f$ of order $m$. 
\end{definition}

Since $\tilde{P} (s) \in \widetilde{\D}_X [s]$ in the 
equation \eqref{rfe-1} can be rewritten as 
\begin{equation}
\tilde{P} (s)= \frac{1}{G^m} \circ \tilde{Q} (s) 
\circ G^m \qquad ( \tilde{Q} (s) \in \widetilde{\D}_X [s]), 
\end{equation}
in fact the condition on $b(s)$ in Theorem \ref{Main-4} 
is equivalent to the existence of some 
$\tilde{Q} (s) \in \widetilde{\D}_X [s]$ satisfying 
the simpler equation 
\begin{equation}
b(s) f^{s} = \tilde{Q} (s)  f^{s+1} 
\end{equation}
independent of $m \geq 0$. This shows that we have 
\begin{equation}
\tilde{b}_{f,0}^{{\rm mero}}(s) = 
\tilde{b}_{f,1}^{{\rm mero}}(s) = 
\tilde{b}_{f,2}^{{\rm mero}}(s) = \cdots \cdots \cdots.
\end{equation}
Therefore we denote $\tilde{b}_{f,m}^{{\rm mero}}(s)$ 
simply by $\tilde{b}_{f}^{{\rm mero}}(s)$. 
Then by our construction, for any $m \geq 0$ 
our b-function $b_{f,m}^{{\rm mero}}(s)$ is 
divided by the reduced one $\tilde{b}_{f}^{{\rm mero}}(s)$. 
We thus obtain a lower bound 
\begin{equation}
( \tilde{b}_{f}^{{\rm mero}})^{-1}(0)
\subset 
(b_{f,m}^{{\rm mero}})^{-1}(0) \subset \QQ
\end{equation}
for the subset $(b_{f,m}^{{\rm mero}})^{-1}(0) \subset \QQ$. 
Several authors studied (global) b-functions on algebraic 
varieties. In particular, the result in \cite{M-N} 
ensures the existence of b-functions on smooth 
affine varieties. See \cite{A-J-N} for a review on 
this subject. Since for algebraic $X$ and $f= \frac{F}{G}$ 
the variety $X \setminus G^{-1}(0)$ is affine, 
our Theorem \ref{Main-4} could be considered 
as an analytic counterpart of their result 
in a very special case. 

\begin{remark}
It looks that the $\D_X$-modules $\K$ 
and $\K [ \frac{1}{G} ]$ above are regular holonomic, 
but we could not prove it. We conjecture that 
they are regular holonomic. 
\end{remark}

From now on, we consider the special case where 
the meromorphic function $f= \frac{F}{G}$ is 
quasi-homogeneous. More precisely, for 
a local coordinate system $x=(x_1,x_2, \ldots, x_n)$ 
of $X$ such that $x_0= \{ x=0 \}$, we assume 
that there exist a vector field 
$v= \sum_{i=1}^n w_i x_i \partial_{x_i} \in \D_X$ 
($w=(w_1,w_2, \ldots, w_n) \in \ZZ_{\geq 0}^n 
\setminus \{ 0 \}$ is a weight vector) and 
$d_1,d_2 \in \ZZ_{>0}$ with $d:=d_1-d_2 \not= 0$ 
such that 
\begin{equation}
vF=d_1 \cdot F, \quad vG=d_2 \cdot G, 
\quad vf=d \cdot f \not= 0.
\end{equation}
Let us calculate $\tilde{b}_{f}^{{\rm mero}}(s)$ 
of such $f$ following the arguments in 
\cite[Section 6.4]{K-book}. 
First, by the condition $vf=d \cdot f$ 
($d \not= 0$) we have isomorphisms 
\begin{equation}
\widetilde{\D}_X [s] f^{s} \simeq 
\widetilde{\D}_X f^{s}, \quad 
\K \Bigl[ \frac{1}{G} \Bigr] \simeq 
\frac{ \widetilde{\D}_X f^{s} }{ \widetilde{\D}_X f^{s+1} }
\end{equation}
and for our reduced b-function $\tilde{b}_{f}^{{\rm mero}}(s)$ 
there exists $\tilde{P} \in \widetilde{\D}_X$ 
such that 
\begin{equation}\label{FEE-1}
\tilde{b}_{f}^{{\rm mero}}(s) f^s= 
\tilde{P} f^{s+1}. 
\end{equation}
In this situation, by the proof of 
\cite[Lemma 6.6]{K-book} we see that 
$\K [ \frac{1}{G} ]$ is a holonomic $\D_X$-module and 
Theorem \ref{Main-4} can be proved also by 
using the trick in the proof of 
\cite[Theorem 6.7]{K-book}. 
If we set $s=-1$ in \eqref{FEE-1} we obtain 
\begin{equation}\label{FEE-2}
\tilde{b}_{f}^{{\rm mero}}(-1)= f \tilde{P}(1). 
\end{equation}
Restricting this equality to the subset 
$F^{-1}(0) \setminus G^{-1}(0) \subset 
X \setminus G^{-1}(0)$ we see that 
$\tilde{b}_{f}^{{\rm mero}}(-1)=0$. 
Namely, for a non-zero polynomial 
$\tilde{\beta}_{f}^{{\rm mero}}(s) \in \CC [s]$ 
we have 
\begin{equation}
\tilde{b}_{f}^{{\rm mero}}(s)= (s+1)  \cdot 
\tilde{\beta}_{f}^{{\rm mero}}(s). 
\end{equation}
On the other hand, by \eqref{FEE-2} we 
have $\tilde{P}(1)=0$ and hence 
$\tilde{P} \in \sum_{i=1}^n \widetilde{\D}_X 
\partial_{x_i}$. Namely there exist 
$\tilde{Q}_i \in \widetilde{\D}_X$ ($1 \leq i \leq n$) 
such that $\tilde{P} = \sum_{i=1}^n \tilde{Q}_i 
\partial_{x_i}$. Moreover, if we set 
\begin{equation}
f_i:=f_{x_i}= \frac{\partial f}{\partial x_i}
= \frac{F_{x_i}G-FG_{x_i}}{G^2} \quad 
(1 \leq i \leq n),
\end{equation}
then we have 
\begin{equation}
\partial_{x_i}f^{s+1}=(s+1)f_if^s 
 \quad (1 \leq i \leq n). 
\end{equation}
Therefore we obtain 
\begin{equation}\label{FEE-3}
\tilde{\beta}_{f}^{{\rm mero}}(s) f^s
= \sum_{i=1}^n \tilde{Q}_i f_if^s. 
\end{equation}
Conversely, for $\tilde{Q}_i \in \widetilde{\D}_X$ 
($1 \leq i \leq n$) satisfying this equality 
the differential operator $\tilde{P} = 
\sum_{i=1}^n \tilde{Q}_i 
\partial_{x_i} \in \widetilde{\D}_X$ 
satisfies the one \eqref{FEE-1}. 
Consequently, our $\tilde{\beta}_{f}^{{\rm mero}}(s) 
\not= 0$ is the minimal polynomial 
$b(s) \in \CC [s]$ satisfying the condition 
$b(s)f^s \in \sum_{i=1}^n \widetilde{\D}_X f_if^s$. 
Since we have 
\begin{equation}
v(f_if^s)=(d-w_i+d \cdot s)(f_if^s) \quad (1 \leq i \leq n),
\end{equation}
$\sum_{i=1}^n \widetilde{\D}_X f_if^s$ is a 
$\widetilde{\D}_X [s]$-submodule of 
$\widetilde{\D}_X f^s \simeq \widetilde{\D}_X [s] f^s$. Set $h_i:= 
F_{x_i}G-FG_{x_i}=G^2f_i\in \sho_X$ 
($1 \leq i \leq n$). Then the $\widetilde{\D}_X$-module 
\begin{equation}
\widetilde{\R}:= 
\frac{ \widetilde{\D}_X f^{s} }{\sum_{i=1}^n 
\widetilde{\D}_X f_if^{s} } 
\simeq 
\frac{ \widetilde{\D}_X f^{s} }{\sum_{i=1}^n 
\widetilde{\D}_X h_if^{s} } 
\end{equation}
has an action of $s$ and the minimal polynomial 
of $s$ on it is equal to $\tilde{\beta}_{f}^{{\rm mero}}(s)$. 

\begin{proposition}\label{PROP-1} 
Let $f= \frac{F}{G}$ be as above and assume moreover 
that $f^{-1}(0)=F^{-1}(0) \setminus G^{-1}(0) 
\subset X \setminus G^{-1}(0)$ is smooth. 
Then $\widetilde{\R}=0$ and 
$\tilde{b}_{f}^{{\rm mero}}(s)=s+1$. 
\end{proposition}

\begin{proof}
Let us consider the coherent $\D_X$-module 
\begin{equation}
\SH := 
\frac{ \D_X }{\sum_{i=1}^n \D_X h_i} 
\end{equation}
and its localization 
\begin{equation}
\widetilde{\SH}:= \SH \Bigl[ \frac{1}{G} \Bigr] 
\simeq 
\frac{ \widetilde{\D}_X }{\sum_{i=1}^n \widetilde{\D}_X h_i}. 
\end{equation}
Note that $\widetilde{\R}$ is a quotient 
of $\widetilde{\SH}$. Since $f= \frac{F}{G}$ 
is quasi-homogeneous of degree $d=d_1-d_2 
\not= 0$, there is no singular point of 
$f$ in $X \setminus (F^{-1}(0) \cup G^{-1}(0))$. 
Then by the smoothness of $f^{-1}(0)=F^{-1}(0) 
\setminus G^{-1}(0) 
\subset X \setminus G^{-1}(0)$, we have 
\begin{equation}
{\rm Sing} \ f = \{ x \in X \setminus G^{-1}(0) \ | \ 
h_1(x)=h_2(x)= \cdots =h_n(x)=0 \} = \emptyset. 
\end{equation}
This implies that the support of the coherent $\D_X$-module 
$\SH$ is contained in $G^{-1}(0) \subset X$. Then by 
Hilbert's nullstellensatz we get 
$\widetilde{\SH}=0$ and hence $\widetilde{\R}=0$. 
\end{proof}

By using \cite[Theorem 3.3 and Corollary 3.4]{N-T} we can 
construct many examples of $f= \frac{F}{G}$ 
satisfying the conditions in Proposition 
\ref{PROP-1} and having a monodromy eigenvalue 
$\not= 1$ at the point $x_0 \in X$. By Theorem \ref{Main-2} 
for such $f$ we thus obtain 
\begin{equation}
b_{f,m}^{{\rm mero}}(s) \not= 
\tilde{b}_{f}^{{\rm mero}}(s)=s+1 
\quad (m \geq 2 \dim X). 
\end{equation}
Namely, in the situation of Proposition 
\ref{PROP-1} for the reduced b-function 
$\tilde{b}_{f}^{{\rm mero}}(s)$ captures 
only the tiny (trivial) part $s+1$ of 
$b_{f,m}^{{\rm mero}}(s)$ for $m \geq 2 \dim X$.

\section{Multiplier ideals for meromorphic functions}\label{sec:s5}

In this section, we define multiplier ideal sheaves 
for the meromorphic function $f= \frac{F}{G}$ 
and study their basic properties. 
Recall that multiplier ideals for 
holomorphic functions were introduced by Nadel \cite{N}. 
For their precise properties, we refer to 
the excellent book \cite{L} by Lazarsfeld. 
For the meromorphic function $f= \frac{F}{G}$ 
we define them as follows. Denote by 
${\rm L}^1_{{\rm loc}}$ the set of locally 
integrable functions on $X$. 

\begin{definition}\label{Multi-ideal}
For a positive real number $\alpha >0$ we define 
an ideal $\II (X, f)_{\alpha} \subset \sho_X$ of 
$\sho_X$ by 
\begin{equation}
\II (X, f)_{\alpha}:= \Bigl\{ h \in \sho_X \ | \ 
\frac{|h|^{2}}{|f|^{2 \alpha}} = 
\frac{|h|^{2} \cdot |G|^{2 \alpha}}{|F|^{2 \alpha}} 
\in {\rm L}^1_{{\rm loc}} \Bigr\}
\end{equation}
and call it the multiplier ideal of $f$ 
of order $\alpha >0$. 
\end{definition}

Let $\pi : Y \longrightarrow X$ be a 
resolution of singularities 
of the divisor $D=
F^{-1}(0) \cup G^{-1}(0) \subset X$ as in 
Section \ref{sec:s4}. Here we assume moreover 
that the meromorphic function 
$g= \frac{F \circ \pi}{G \circ \pi}$ has no 
point of indeterminacy on the whole $Y$. 
Such a resolution $\pi : Y \longrightarrow X$ 
always exists. Let ${\rm div} \ g$ be the divisor on $Y$ 
defined by $g$. Then there exist two 
effective divisors $({\rm div} \ g)_+$ and 
$({\rm div} \ g)_-$ such that 
\begin{equation}
{\rm div} \ g = 
({\rm div} \ g)_+ - ({\rm div} \ g)_-. 
\end{equation}
By our assumption, their supports, which we 
denote by $g^{-1}(0)$ and $g^{-1}( \infty )$ 
respectively, are disjoint from each other. 
By using such a resolution $\pi : Y \longrightarrow X$, 
we can easily see that for $\alpha^{\prime} > \alpha >0$ 
we have $\II (X, f)_{\alpha^{\prime}} \subset 
\II (X, f)_{\alpha}$. Then, as in the case 
where $f$ is holomorphic, we can define 
the jumping numbers of the multiplier ideals 
$\{ \II (X, f)_{\alpha} \}_{\alpha > 0}$. 
In the situation as above, we have $g^{-1}( \infty ) \subset 
(G \circ \pi )^{-1}(0)$ but 
$g: Y \setminus (G \circ \pi )^{-1}(0) 
\longrightarrow \CC$ can be extended to a 
holomorphic function 
$\tilde{g} : Y \setminus g^{-1}( \infty )
\longrightarrow \CC$. Let 
\begin{equation}
\iota_{\tilde{g}} : Y \setminus g^{-1}( \infty )
\longrightarrow Y \times \CC \qquad 
(y \longmapsto (y, \tilde{g} (y)))
\end{equation} 
be the graph embedding defined by $\tilde{g}$. 
From now, we shall use the terminologies of 
mixed Hodge modules. For example, regarding 
the holonomic $\D_{X \times \CC}$-module $\M$ 
as a mixed Hodge module on $X \times \CC$, 
for $\alpha \in \QQ$ and $p \in \ZZ$ we set 
\begin{equation}
F_pV_{\alpha} \M = F_p \M \cap V_{\alpha} \M. 
\end{equation} 
We denote the normal crossing divisor 
$(G \circ \pi )^{-1}(0)$ in $Y$ by $E$ and 
consider the regular holonomic $\D_Y$-module 
$\sho_Y(*E)$ as a mixed Hodge module. 
Then its Hodge filtration $\{ F_p \sho_Y(*E) \}_{p \in \ZZ}$ 
satisfies the condition 
\begin{equation}
F_p \sho_Y(*E) \simeq 0 \quad (p<0), \qquad 
F_0 \sho_Y(*E) \not=0. 
\end{equation} 
Moreover $F_0 \sho_Y(*E) \subset \sho_Y(*E)$ is 
the subsheaf of $\sho_Y(*E)$ consisting of 
meromorphic functions on $Y$ having 
poles of order $\leq 1$ only along $E \subset Y$. 
See Mustata-Popa \cite[Chapter D]{M-P} for the details 
about the Hodge filtration of $\sho_Y(*E)$. 
We denote the restriction of 
$F_0 \sho_Y(*E) \simeq \sho_Y(E)$ to 
$Y \setminus g^{-1}( \infty )$ 
simply by $\E$. Then the following 
propositon can be proved just by following 
the arguments in Budur-Saito \cite{B-S} 
(see also \cite[Sections 3 and 4]{Budur-N} for more precise 
explanations). We set $Y^{\circ}:=
Y \setminus g^{-1}( \infty )$ and 
\begin{equation}
\varpi :=( \pi \times 
{\rm id}_{\CC} ) \circ \iota_{\tilde{g}}
 : Y^{\circ}
\longrightarrow X \times \CC \qquad 
(y \longmapsto ( \pi (y) , \tilde{g} (y))).
\end{equation}
Let $K_{Y/X}$ be the relative canonical 
divisor of $\pi : Y \longrightarrow X$. 

\begin{proposition}\label{PROP-2} 
Let $\alpha >0$ be a positive real number. 
Then for $0< \varepsilon \ll 1$ 
there exists an isomorphism 
\begin{equation}
F_1V_{- \alpha} \M \simeq \varpi_* 
\Bigl\{ \sho_{Y^{\circ}}(K_{Y/X}) \otimes 
\E \cap \sho_{Y^{\circ}} 
\Bigl(- \lfloor ( \alpha - \varepsilon ) ({\rm div} \ g)_+ 
\rfloor \Bigr) \Bigr\}. 
\end{equation} 
\end{proposition}
Now let ${\rm pr}_X: X \times \CC \longrightarrow 
X$ be the projection. Then there exists an 
injective homomrphism of sheaves 
\begin{equation}
\gamma : \sho_X  \longrightarrow ({\rm pr}_X)_* \M 
\qquad (h \longmapsto h \cdot \sigma_0), 
\end{equation} 
by which we regard $\sho_X$ as a subsheaf of 
$({\rm pr}_X)_* \M$. By Proposition 
\ref{PROP-2} and the local 
integrability condition in Definition \ref{Multi-ideal}, 
we obtain the following analogue for 
meromorphic functions of 
Budur-Saito \cite[Theorem (0.1)]{B-S} 
(see also \cite{Budur-N} for the details). 
Note that by Proposition \ref{PROP-2} 
for any $\alpha >0$ we have 
\begin{equation}
\sho_X \cap ({\rm pr}_X)_* 
F_1V_{- \alpha} \M = 
\sho_X \cap ({\rm pr}_X)_* 
V_{- \alpha} \M.
\end{equation} 

\begin{theorem}\label{Main-5} 
Let $\alpha >0$ be a positive real number. 
Then we have 
\begin{equation}
\II (X, f)_{\alpha} = \sho_X \cap ({\rm pr}_X)_* 
V_{< - \alpha} \M = 
\{ h \in \sho_X \ | \ 
{\rm ord}_{\{ t=0 \}}(h \cdot \sigma_0) > 
\alpha -1 \}. 
\end{equation} 
\end{theorem}
By this theorem and 
\begin{equation}
{\rm ord}_{\{ t=0 \}}(h \cdot \sigma_0) \subset 
\bigcup_{i=0,1,2, \ldots}
\Bigl\{ {\rm ord}_{\{ t=0 \}}( \sigma_0) +i \Bigr\} 
\qquad (h \in \sho_X)
\end{equation} 
(see the proof of Lemma \ref{Lemma-4}), 
we immediately obtain 
the following generalization of 
the celebrated theorem of Ein-Lazarsfeld-Smith-Varolin 
\cite{E-L-S-V} to meromorphic functions. 

\begin{corollary}\label{Cor-jump} 
The jumping numbers of the multiplier ideals 
$\{ \II (X, f)_{\alpha} \}_{\alpha >0}$ are 
contained in the set 
\begin{equation}
\bigcup_{i=0,1,2, \ldots}
\Bigl\{ -(b_{f,0}^{{\rm mero}})^{-1}(0)+i \Bigr\}
\subset \QQ_{>0}. 
\end{equation}
Moreover the minimal jumping number $\alpha >0$ is 
equal to the negative of the largest root 
of $b_{f,0}^{{\rm mero}}(s)$. 
\end{corollary}


\begin{thebibliography}{99}

\bibitem{A-G-L-N} 
{\`A}lvarez Montaner, J., Gonz{\'a}lez Villa, M., 
Le{\'o}n-Cardenal, E. and N{\'u\~n}ez-Betancourt, L., 
\emph{Bernstein-Sato polynomial and related invariants 
for meromorphic functions}, 
 arXiv:2112.08492v1, preprint.

\bibitem{A-J-N} 
{\`A}lvarez Montaner, J., Jeffries, J. and N{\'u\~n}ez-Betancourt, L., 
\emph{Bernstein-Sato polynomials in commutative algebra}, 
arXiv:2106.08830v3, preprint.

\bibitem{B-1} 
Bahloul, R. 
\emph{Algorithm for computing Bernstein-Sato ideals 
associated with a polynomial mapping},
J. Symbolic Comput. 32, No.6 (2001): 643-662.

\bibitem{B} 
Bahloul, R. 
\emph{D\'emonstration constructive de l'existence de polyn\^omes 
de Bernstein-Sato pour plusieurs fonctions analytiques},
Compositio Math. 141, No.1 (2005): 175-191.

\bibitem{B-O} 
Bahloul, R. and Oaku, T., 
\emph{Local Bernstein-Sato ideals: algorithm and examples},
J. Symbolic Comput., 45, No.1 (2010): 46-59.

\bibitem{Budur}
Budur, N. \emph{Bernstein-Sato polynomials and generalizations}, 
lecture notes, Rolduc Abbey (2013). 

\bibitem{Budur-N}
Budur, N. \emph{Applications of mixed Hodge modules: multiplier ideals}, 
Notes from a lecture at the workshop ``Mixed Hodge Modules 
and their Applications", 19-23 Aug. 2013, Clay 
Mathematics Institute, Oxford. 

\bibitem{B-M-S} 
Budur, N., Mustata, M. and Saito, M., 
\emph{Bernstein-Sato polynomials of arbitrary varieties}, 
Compositio Math. 142, No.3 (2006): 779-797.

\bibitem{B-S} 
Budur, N. and Saito, M., 
\emph{Multiplier ideals, V-filtration, and spectrum}, 
J. Algebraic Geom. 14, No.2 (2005): 269-282.

\bibitem{B-V-W-Z} 
Budur, N., van der Veer, R. Wu, L. and Zhou, P., 
\emph{Zero loci of Bernstein-Sato ideals},
Invent. Math. 225, No.1 (2021): 45-72.

\bibitem{D-L}
Denef, J. and Loeser, F., \emph{Caract\'eristique d'Euler-Poincar\'e, 
fonctions z\^eta locales et modifications 
analytiques}, J. Amer. Math. Soc., 5 (1992): 705-720.

\bibitem{Dimca}
Dimca, A. \emph{Sheaves in topology}, 
Universitext, Springer-Verlag, Berlin, 2004.

\bibitem{E-L-S-V}
Ein, L., Lazarsfeld, R., Smith, K. and Varolin, R., 
\emph{Jumping coefficients of multiplier ideals}, 
Duke Math. J., 123, No. 3 (2004): 469-506.  

\bibitem{G-L} 
Gonz{\'a}lez Villa, M. and Lemahieu, A., 
\emph{The monodromy conjecture for plane meromorphic germs}, 
Bull. Lond. Math. Soc., 46, No.3 (2014): 441-453.

\bibitem{G-L-M}
Gusein-Zade, S., Luengo, I. and Melle-Hern\'andez, A. 
\emph{Zeta-functions for germs of meromorphic 
functions and Newton diagrams}, 
Funct. Anal. Appl., 32, No. 2 (1998): 26-35.  

\bibitem{G} 
Gyoja, A., \emph{Bernstein-Sato's polynomial for 
several analytic functions}, 
J. Math. Kyoto Univ., 33, No.2 (1993): 399-411.

\bibitem{H-T-T}
Hotta, R., Takeuchi, K. and Tanisaki, T. 
\emph{D-modules, perverse 
sheaves, and representation theory}, 
Birkh{\"a}user Boston, 2008.

\bibitem{K-1} 
Kashiwara, M., \emph{B-functions and holonomic systems, 
Rationality of roots of B-functions}, 
Invent. Math. 38, No.1 (1976): 33-53.

\bibitem{K-2} 
Kashiwara, M., \emph{Vanishing cycle sheaves and holonomic 
systems of differential equations}, 
Lecture Notes in Math., Vol. 1016, 
pp. 134-142, Springer, Berlin (1983).

\bibitem{K-book}
Kashiwara, M.. \emph{D-modules and 
microlocal calculus}, translated from the 2000 
Japanese original by Mutsumi Saito., 
American Mathematical Society, 2003.

\bibitem{K-S}
Kashiwara, M. and Schapira, P. 
\emph{Sheaves on manifolds}, Springer-Verlag, 1990.

\bibitem{L}
Lazarsfeld, R., 
\emph{Positivity in algebraic geometry. II., 
Positivity for vector bundles, and multiplier ideals}, 
Springer-Verlag, 2004. 

\bibitem{M} 
Malgrange, B., \emph{Polyn\^omes de Bernstein-Sato 
et cohomologie \'evanescente}, 
Ast\'erisque, 101 (1983).

\bibitem{M-N}
Mebkhout, Z. and Narv\'aez-Macarro, L. \emph{La th\'eorie du 
polyn\^ome de Bernstein-Sato pour les alg\`ebres de 
Tate et de Dwork-Monsky-Washnitzer}, 
Ann. Sci. \'Ecole Norm. Sup. 24, No.2 (1991): 227-256.

\bibitem{M-S}
Mebkhout, Z. and Sabbah, C. \emph{Section III.4, 
D-modules et cycles \'evanescents., Le formalisme 
des six op\'erations de Grothendieck pour les D-modules 
coh\'erents}, 201-239. 
Travaux en cours 35. Paris: Hermann, 1989.

\bibitem{Milnor}
Milnor, J. \emph{Singular 
points of complex hypersurfaces}, Princeton 
University Press, 1968.

\bibitem{M-P} 
Mustata, M. and Popa, M., 
\emph{Hodge ideals}, 
Mem. Amer. Math. Soc. 262 (2019): no. 1268.

\bibitem{N} 
Nadel, A. M., 
\emph{Multiplier ideal sheaves and K\"ahler-Einstein metrics 
of positive scalar curvature}, 
Ann. Math. 132, No.3 (1990): 549-596.

\bibitem{N-T} 
Nguyen, T. T. and Takeuchi, K., 
\emph{Meromorphic nearby cycle functors and 
monodromies of meromorphic functions 
(with Appendix by T. Saito)}, 
to appear in Revista Matematica Complutense.

\bibitem{O} 
Oaku, T., \emph{An algorithm of computing b-functions}, 
Duke Math. J., 87, No.1 (1997): 115-132.

\bibitem{O-T} 
Oaku, T. and Takayama, N.  
\emph{An algorithm for de Rham cohomology groups of the complement 
of an affine variety via D-module computation}, 
J. Pure Appl. Algebra, 139, No.1-3 (1999): 201-233.

\bibitem{Raibaut}
Raibaut, M. \emph{Motivic Milnor fibers of a rational function}, 
Revista Matematica Complutense, 
26, (2013): 705-734. 

\bibitem{Sabbah-1}
Sabbah, C. \emph{Proximit\'e \'evanescente. I.}, 
Compositio Math., 62, No. 3 (1987): 283-328. 

\bibitem{Sabbah-2}
Sabbah, C. \emph{Proximit\'e \'evanescente. II.}, 
Compositio Math., 64, No. 2 (1987): 213-241. 

\bibitem{Sabbah-3}
Sabbah, C. \emph{Introduction to mixed Hodge modules}, 
personal note of a mini course, Angers (2019). 

\bibitem{S} 
Saito, M. \emph{On microlocal b-functions}, 
Bull. Soc. Math. France, 122 (1994): 163-184.

\bibitem{U-C} 
Ucha, J. M. and Castro-Jim\'enez, F. J. 
\emph{On the computation of Bernstein-Sato ideals},
J. Symbolic Comput. 37, No.5 (2004): 629-639.

\bibitem{V-Z} 
Veys, W. and Zuniga-Galindo, W. A., 
\emph{Zeta functions and oscillatory 
integrals for meromorphic functions}, 
Advances in Math., 311 (2017): 295-337.

\end{thebibliography}
\end{document}